\documentclass{article}

\newtheorem{theorem}{Theorem}

\newtheorem{conjecture}[theorem]{Conjecture}
\newtheorem{corollary}[theorem]{Corollary}

\newenvironment{proof}[1][Proof]{\noindent\textbf{#1.} }{\ \rule{0.5em}{0.5em}}
\input{tcilatex}
\begin{document}

\title{An Interesting Identity}
\author{Brett Pansano \\
Northwest Arkansas Community College}
\maketitle

\begin{abstract}
This purpose of this paper is to note an interesting identity derived from
an integral in Gradshteyn and Ryzhik using techniques from George
Boros'(deceased) Ph.D thesis. The idenity equates a sum to a product by
evaluating an integral in two different ways. A more general form of the
idenity is left for further investigation.
\end{abstract}

\section{Introduction}

\begin{theorem}
Let n be non-negative integer. Then 
\[
\sum_{k=0}^{n}\left( -1\right) ^{k}\binom{n}{k}\frac{1}{nk+n-1}%
=n^{n}n!\dprod\limits_{k=0}^{n}\frac{1}{nk+n-1}. 
\]
\end{theorem}

\begin{proof}
We will evaluate the integral 
\[
\int_{0}^{1}\frac{\left( 1-x\right) ^{n}}{x^{\frac{1}{n}}}dx 
\]%
in two different ways. We start with 
\[
\left( 1-x\right) ^{n}=\sum_{k=0}^{n}\left( -1\right) ^{k}\binom{n}{k}x^{k}, 
\]%
which follows directly from the binomial theorem. Dividing both sides by x$^{%
\frac{1}{n}}$ yields 
\[
\frac{\left( 1-x\right) ^{n}}{x^{\frac{1}{n}}}=\sum_{k=0}^{n}\left(
-1\right) ^{k}\binom{n}{k}x^{k-\frac{1}{n}}, 
\]%
so that 
\[
\int_{0}^{1}\frac{\left( 1-x\right) ^{n}}{x^{\frac{1}{n}}}%
dx=\sum_{k=0}^{n}\left( -1\right) ^{k}\binom{n}{k}\int_{0}^{1}x^{k-\frac{1}{n%
}}dx=\sum_{k=0}^{n}\left( -1\right) ^{k}\binom{n}{k}\frac{n}{nk+n-1}. 
\]%
But the integral $\int_{0}^{1}\frac{\left( 1-x\right) ^{n}}{x^{\frac{1}{n}}}%
dx$ can be evaluated using the fact that 
\[
\int_{0}^{1}x^{m-1}\left( 1-x\right) ^{n-1}dx=B\left( m,n\right) 
\]%
where $B\left( m,n\right) $ is the beta function. We get 
\[
\int_{0}^{1}\frac{\left( 1-x\right) ^{n}}{x^{\frac{1}{n}}}dx=B\left( n+1,1-%
\frac{1}{n}\right) =\frac{\Gamma \left( n+1\right) \Gamma \left( 1-\frac{1}{n%
}\right) }{\Gamma \left( n+2-\frac{1}{n}\right) }, 
\]%
since 
\[
B(m,n)=\frac{\Gamma \left( m\right) \Gamma \left( n\right) }{\Gamma \left(
m+n\right) }, 
\]%
where $\Gamma $ is the gamma function.

We next note that 
\begin{eqnarray*}
\Gamma \left( n+2-\frac{1}{n}\right) &=&\left( n+1-\frac{1}{n}\right) \Gamma
\left( n+1-\frac{1}{n}\right) \\
&=&\left( n+1-\frac{1}{n}\right) \left( n-\frac{1}{n}\right) \Gamma \left( n-%
\frac{1}{n}\right) \\
&=&...=\left( n+1-\frac{1}{n}\right) \left( n-\frac{1}{n}\right) \left( n-1-%
\frac{1}{n}\right) \left( n-2-\frac{1}{n}\right) \\
&&\cdot \cdot \cdot \left( n-\frac{1}{n}\right) \Gamma \left( 1-\frac{1}{n}%
\right) ,
\end{eqnarray*}%
so that (since $\Gamma \left( n+1\right) =n!)$ 
\begin{eqnarray*}
&&\int_{0}^{1}\frac{\left( 1-x\right) ^{n}}{x^{\frac{1}{n}}}dx \\
&=&\frac{n!}{\left( n+1-\frac{1}{n}\right) \left( n-\frac{1}{n}\right)
\left( n-1-\frac{1}{n}\right) \left( n-2-\frac{1}{n}\right) \cdot \cdot
\cdot \left( 1-\frac{1}{n}\right) } \\
&=&\frac{n!}{\dprod\limits_{k=1}^{n+1}\left( k-\frac{1}{n}\right) }%
=\dprod\limits_{k=1}^{n+1}\left( nk-1\right) .
\end{eqnarray*}%
Thus, 
\[
\sum_{k=0}^{n}\left( -1\right) ^{k}\binom{n}{k}\frac{n}{nk+n-1}=\frac{%
n^{n+1}n!}{\dprod\limits_{k=1}^{n+1}\left( nk-1\right) }=n^{n}n!\dprod%
\limits_{k=0}^{n}\frac{1}{nk+n-1}. 
\]%
so that,%
\[
\sum_{k=0}^{n}\left( -1\right) ^{k}\binom{n}{k}\frac{1}{nk+n-1}=\frac{n^{n}n!%
}{\dprod\limits_{k=1}^{n+1}\left( nk-1\right) }=n^{n}n!\dprod%
\limits_{k=0}^{n}\frac{1}{nk+n-1}. 
\]
\end{proof}

\begin{corollary}
Let a be a real number. Then
\end{corollary}

\begin{enumerate}
\item 
\[
\sum_{k=0}^{n}\left( -1\right) ^{k}\binom{n}{k}\frac{1}{nk+an-1}%
=n^{n}n!\dprod\limits_{k=0}^{n}\frac{1}{nk+an-1} 
\]

\item 
\[
\sum_{k=0}^{n}\left( -1\right) ^{k}\binom{n}{k}\frac{1}{ank+n-1}=\left(
an\right) ^{n}n!\dprod\limits_{k=0}^{n}\frac{1}{ank+n-1} 
\]

\item 
\[
\sum_{k=0}^{n}\left( -1\right) ^{k}\binom{n}{k}\frac{1}{ank+an-1}=\left(
an\right) ^{n}n!\dprod\limits_{k=0}^{n}\frac{1}{ank+n-1} 
\]
\end{enumerate}

\begin{proof}
This is clear.
\end{proof}

The following identity was confirmed via Mathematica

$\ \frac{d}{da}\dprod\limits_{k=0}^{3}\frac{1}{nk+an-1}\allowbreak
\allowbreak $ : $\frac{n^{4}}{\left( an-1\right) ^{2}\left( n+an-1\right)
^{2}\left( 2n+an-1\right) ^{2}\left( 3n+an-1\right) ^{2}}\allowbreak $ by
noting that 
\[
\frac{d}{da}\frac{1}{nk+an-1}=-\frac{1}{\left( nk+an-1\right) ^{2}}n. 
\]%
The following might be interesting to investigate:

\begin{conjecture}
Let n be a natural number. Suppose f and g are differentiable functions.
Then 
\[
\sum_{k=0}^{n}\left( -1\right) ^{k}\binom{n}{k}\frac{1}{f\left( n\right)
k+g\left( n\right) }=n^{n}n!\dprod\limits_{k=0}^{n}\frac{1}{f\left( n\right)
k+g\left( n\right) }.
\]
\end{conjecture}


\begin{thebibliography}{G. Boros and V. Moll, Irresistible Integrals,
Symbolics, Analysis, and Experiments in the Evaluation of Integrals,
Cambridge University Press, 2004}
\bibitem[I.S. Gradshteyn and I.M. Ryzhik, Table of Integrals, Series, and
Products, Academic Press, INC., 1980]{1} 

\bibitem[G. Boros and V. Moll, Irresistible Integrals, Symbolics, Analysis,
and Experiments in the Evaluation of Integrals, Cambridge University Press,
2004]{2} 
\end{thebibliography}
\end{document}